\documentclass[a4paper,12pt]{article}

\usepackage{amsmath,amssymb,amsthm, mathrsfs}
\usepackage{latexsym}
\usepackage{graphics}
\usepackage{hyperref}
\usepackage{graphicx,psfrag}
\usepackage[bf, small]{titlesec}
\usepackage{authblk} 
\usepackage{pgfpages,tikz,pgfkeys,pgfplots} 
\usetikzlibrary{arrows,positioning,matrix,fit,backgrounds,shapes}
\usepackage{soul}

\usepackage{lineno}

\theoremstyle{plain}
\newtheorem{Thm}{Theorem}[section]
\newtheorem{Lem}[Thm]{Lemma}
\newtheorem{Prop}[Thm]{Proposition}
\newtheorem{Cor}[Thm]{Corollary}

\newtheorem{Conj}[Thm]{Conjecture}

\theoremstyle{definition}
\newtheorem{Defi}[Thm]{Definition}
\newtheorem{Rem}[Thm]{Remark}

\usepackage{standalone}
\usepackage{pgfpages,tikz,pgfkeys,pgfplots}
\usetikzlibrary{arrows,positioning,matrix,fit,backgrounds,shapes,shapes.geometric, intersections}

\usetikzlibrary{shapes.callouts,decorations.pathmorphing} 
\usetikzlibrary{shadows} 
\usepackage{calc} 
\usetikzlibrary{decorations.markings} 
\tikzstyle{vertex}=[circle, draw, inner sep=0pt, minimum size=6pt] 
\newcommand{\vertex}{\node[vertex]}
\usetikzlibrary{arrows,matrix} 

\newcommand{\RR}{\mathbb{R}} 
\newcommand{\NN}{\mathbb{N}} 

\newcommand{\AAA}{\mathcal{A}} 
\newcommand{\NNN}{\mathcal{N}}

\newcommand{\GGG}{\mathcal{G}}

\title{On weak majority dimensions of digraphs}

\date{}
\author[1]{\small Soogang Eoh}
\author[1]{\small Suh-Ryung Kim}
\affil[1]{\footnotesize Department of Mathematics Education, Seoul National University, Seoul 08826}
\affil[ ]{\footnotesize\textit{mathfish@snu.ac.kr, srkim@snu.ac.kr}}

\begin{document}

\maketitle

\begin{abstract}
In this paper, we introduce the notion of the weak majority dimension of a digraph which is well-defined for any digraph.
We first study properties shared by the weak dimension of a digraph and show that a weak majority dimension of a digraph can be arbitrarily large.
Then we present a complete characterization of digraphs of weak majority dimension $0$ and $1$, respectively, and show that every digraph with weak majority dimension at most two is transitive.
Finally, we compute the weak majority dimensions of directed paths and directed cycles and pose  open problems.
\end{abstract}



\section{Introduction}

Throughout this paper, we only deal with digraphs whose underlying graphs are simple.

For a positive integer $n$, we denote the set $\{1,2,\ldots,n\}$ by $[n]$.
For a nonnegative integer $d$ and a point $u$ in $\RR^d$, we denote the $i$th coordinate of $u$ by $[u]_i$ so that $u = ([u]_1, [u]_2, \ldots, [u]_d)$.
In the rest of this paper, we assume that $d$ is a nonnegative integer unless otherwise stated.

While studying the  partial order competition dimension of a graph (see~\cite{choi2016competition} and \cite{choi2017partial}), we have come up with an idea of defining the dimension of a digraph.
In ecosystem, we may represent each species as a vector in $\RR^d$ in such a way that the $d$ components of each vector indicate the average speed, the average weight, the average intelligence, and so on, for the species corresponding to the vector.
Then a species A may be regarded as being superior to a species B if $|\{ i \in [d] \mid [v]_i > [w]_i \}| > |\{i \in [d] \mid [w]_i > [v]_i \}|$ where $v$ and $w$ are the vector corresponding to A and B, respectively.

On the other hand, we let each of voters grant marks represented by real numbers to each of alternatives for evaluation.
To be more precise, let $d$ be the number of voters.
For two alternatives $x$ and $y$, let $x_i$ and $y_i$ be the marks given by the voter $i$.
Then we may correspond $x$ to $(x_1,x_2,\ldots, x_d)\in \mathbb{R}^d$ and $y$ to $(y_1,y_2,\ldots, y_d) \in \mathbb{R}^d$ (by convention, the zero-dimensional Euclidean space $\RR^0$ consists of a single point $0$).
For notational convenience, we write $x=(x_1,x_2,\ldots, x_d)\in \mathbb{R}^d$ and $y=(y_1,y_2,\ldots, y_d) \in \mathbb{R}^d$.
For $x$ and $y$ in $\RR^d$, we let $$\GGG_{x>y}=\{i \in [d] \mid x_i > y_i \}.$$
We define a \emph{weak majority relation} $\prec$ on $\RR^d$ by
$$y \succ x \Leftrightarrow |\GGG_{y>x}|-|\GGG_{x>y}|>0.$$
If $x \succ y$ or $y \succ x$, then we say that $x$ and $y$ are \emph{comparable} in $(\RR^d, \succ)$.
Otherwise, we say that $x$ and $y$ are \emph{incomparable} in $(\RR^d, \succ)$.
Therefore $x$ and $y$ are comparable in $(\RR^d, \succ)$ if and only if $|\GGG_{x>y}| \neq |\GGG_{y>x}|$.

For a nonnegative integer $d$, we say that a digraph $D$ is \emph{$\RR^d$-realizable} if we may assign a map $f:V(D) \to \RR^d$ for some nonnegative integer $d$ so that $(x, y) \in A(D)$ $\Leftrightarrow$ $f(x) \succ f(y)$.
We call such a map $f$ an \emph{$\RR^d$-realizer} of $D$.

Suppose that $D$ is $\RR^d$-realizable for some nonnegative integer $d$ and $r$ is an integer greater than $d$.
Then there exists an $\RR^d$-realizer $f$ of $D$.
We may extend the codomain of $f$ by adjoining $r-d$ components with zero at the end of $f$-value of each vertex to obtain an $\RR^r$-realizer of $D$.
Thus it is true that
\begin{itemize}
  \item[($\star$)] $D$ is $\RR^d$-realizable for some nonnegative integer $d$, then it is $\RR^r$-realizable for any  positive integer $r$, $r > d$.
\end{itemize}

Now we define the weak majority dimension of a digraph in terms of $\RR^d$-realizability, which shall be mainly studied in this paper.

\begin{Defi}
For a digraph $D$,
the \emph{weak majority dimension}, denoted by $\dim(D)$, of $D$ is the minimum nonnegative integer $d$ such that $D$ is $\RR^d$-realizable.
\end{Defi}

Originally, we named a weak majority relation a majority relation.
Upon completing paper, we have searched papers with titles including majority relation just in case in which there might be existing work.
We found that there has been a lot of research done on the ``majority relation'' which is analogous to our majority relation \cite{seedig2015majority}.
Fortunately, our majority relation turns out to be a generalization of the ``majority relation'' in the following sense.

Let $\AAA$ be a set of $m$ \emph{alternatives} and $\NNN:=[n]$ be a set of \emph{voters}.
The preferences of voter $i$ in $\NNN$ are represented by an asymmetric, transitive, and complete relation $\succ_i \subset \AAA \times \AAA$.
We may interpret $(a,b) \in \succ_i$ as the voter $i$ preferring the alternative $a$ to the alternative $b$.

A \emph{preference profile} $R:=(\succ_1, \succ_2, \ldots, \succ_n)$ is an $n$-tuple containing a preference relation $\succ_i$ for each voter $i \in \NNN$.
For a preference profile $R$ and two alternatives $a$ and $b$ in $\AAA$, the \emph{majority margin} $g_R(a,b)$ is defined as difference between the number of voters who prefer $a$ to $b$ and the number of voters who prefer $b$ to $a$, that is,
\[
g_R(a,b) = |\{i \in \NNN \mid a \succ_i b\}| - |\{i \in \NNN \mid b \succ_i a\}|.
\]

The \emph{majority relation} $\succ_R$ of the preference profile $R$ is defined as
\[
a \succ_R b \text{ if and only if } g_R(a,b) > 0.
\]

\begin{Defi}
A \emph{weak preference profile} is a preference profile $(\succeq_1, \succeq_2, \ldots, \succeq_n)$ containing a reflexive, \emph{antisymmetric}, transitive, and complete preference relation $\succeq_i$ for each voter $i$ in $\NNN$.
\end{Defi}

We may encounter a weak preference profile in a real-world situation.
At a tasting event, participants are allowed to be indifferent between two choices.
At a job interview, interviewers are allowed to be indifferent between two interviewees.
In fact, we may regard a weak preference profile as a generalization of preference profile.
We note that
\[
g_R(a,b)=\{i \in \NNN \mid a \succ_i b\}| - |\{i \in \NNN \mid b \succ_i a\}|=\{i \in \NNN \mid a \succeq_i b\}| - |\{i \in \NNN \mid b \succeq_i a\}|.
\]
Therefore the notion of majority margin may be extended to weak preference profiles and we may define the weak majority relation of a weak preference profile in terms of majority margin.

\begin{Defi}
The \emph{weak majority relation} of a weak majority preference profile $R$ is defined by
\[
a \succ_R b \text{ if and only if } g_R(a,b) > 0.
\]
\end{Defi}

Then, by definition, the weak majority relation $\succ_R$ of $R$ may be embedded into a weak majority relation $\succ$ on $\RR^d$.
Conversely, let $d$ be a nonnegative integer and $\succ$ be a weak majority relation on $\RR^d$.
We define $\succeq_i$ on $\RR^d$ by
\[
x \succeq_i y \quad \text{if and only if} \quad x_i \ge y_i
\]
for $x = (x_1,x_2,\ldots, x_d)\in \mathbb{R}^d$ and $y = (y_1,y_2,\ldots, y_d) \in \mathbb{R}^d$.
Then we let $R=(\succeq_1, \ldots, \succeq_d)$.
It is easy to check that $R$ is a weak preference profile and
\[
|\GGG_{x>y}|-|\GGG_{y>x}|= g_R(x,y)= |\{i \in \NNN \mid x \succeq_i y\}| - |\{i \in \NNN \mid y \succeq_i x\}|.
\]

\begin{Rem}
Given a preference profile with $d$ voters, the majority relation on the set of choices may be embedded into $\succ$ restricted to ${\RR^*}^d$ where ${\RR^*}^d$ is a set of $d$-tuples in $\RR^d$ without equal components.
Then we may define the notion of  ${\RR^*}^d$-realizable similarly to $\RR^d$-realizable by restricting the codomain of $f$ to ${\RR^*}^d$.
The majority dimension of a digraph $D$ is defined to be the minimum nonnegative integer $d$ such that $D$ is ${\RR^*}^d$-realizable.
Thus, given a digraph $D$, the weak majority dimension of $D$ is less than or equal to the majority dimension of $D$.
It is known that the majority dimension of a digraph is well-defined.
Therefore the weak majority dimension is well-defined.
\end{Rem}
%

\section{Properties shared by the weak majority dimension of a digraph}

We denote the point $(\min\{[u]_1,[v]_1\}, \ldots, \min\{[u]_d,[v]_d\})$  in $\RR^d$ by $\min\{u,v\}$.
Therefore $[\min\{u,v\}]_i = \min\{ [u]_i, [v]_i\}$ for any $i \in [d]$.

We first show that every digraph is $\RR^d$-realizable for some positive integer $d$.
We need the following lemmas.

\begin{Lem}\label{lem:no arc}
For a digraph $D$, $\dim(D) = 0$ if and only if $D$ is an empty digraph.
\end{Lem}
\begin{proof}
By definition, $\dim(D) = 0$ if and only if there is an $\RR^0$-realizer $f$ of $D$ if and only if  there is a map $f:V(D) \to \{0\}$ such that $(x, y) \in A(D)$ $\Leftrightarrow$ $f(x) \succ f(y)$ if and only if $D$ is an empty digraph.
\end{proof}

\begin{Lem}\label{lem:1 arc dim}
Every digraph having exactly one arc has weak majority dimension at most two.
\end{Lem}
\begin{proof}
Let $(u, v)$ be the arc of a digraph $D$ having exactly one arc.
Then we define a map $f : V(D) \to \RR^2$ as
$$
\text{$[f(u)]_i=i+1$, $[f(v)]_i=i$, $[f(w)]_1=3$, and $[f(w)]_2=1$}
$$ for each $i=1, 2$ and each $w \in V(D) \setminus \{u, v\}$.
It is easy to check that $f$ is an $\RR^2$-realizer of $D$.
\end{proof}

It is known that, for a digraph $D$ and an arc $a$ of $D$, the majority dimension of $D$ is at most two less than the majority dimension of $D-a$.
This inequality also holds for the weak majority dimension of a digraph.

\begin{Prop}\label{prop:add arc}
Let $D$ be a nonempty digraph.
Then $\dim(D) \le \dim(D-a)+2$ for any $a \in A(D)$.
\end{Prop}
\begin{proof}
Let $D$ be a nonempty digraph and let $(u, v)$ be an arc of $D$.
If $|A(D)|=1$, then $\dim(D-(u,v))=0$ by Lemma~\ref{lem:no arc} and, moreover,  by Lemma~\ref{lem:1 arc dim}, $\dim(D) \le 2$.
Thus the proposition statement is true if $|A(D)|=1$.
Now suppose $|A(D)| \ge 2$.
Then $\dim(D-(u,v))=d$ for a positive integer $d$ by Lemma~\ref{lem:no arc}.
Let $f$ be an $\RR^d$-realizer of $D-(u,v)$.
Since the underlying graph of $D$ is simple,  there is no arc between $u$ and $v$ in $D-(u,v)$ and so $f(u)$ and $f(v)$ are incomparable.
Thus $|\GGG_{f(u)>f(v)}|=|\GGG_{f(v)>f(u)}|$.
Now we define a map $g:V(D) \to \RR^{d+2}$ of $D$ out of $f$ as follow:
\begin{itemize}
  \item[(i)] $[g(w)]_i=[f(w)]_i$ for any $w \in V(D)$ and any $i \in [d]$;
  \item[(ii)] $[g(w)]_{d+1}=0$ and $[g(w)]_{d+2}=1$ for any $w \in V(D) \setminus \{u, v\}$;
  \item[(iii)] $[g(u)]_{d+1}=2$, $[g(v)]_{d+1}=1$ and $[g(u)]_{d+2}=[g(v)]_{d+2}=0$.
\end{itemize}
It is easy to check that $g$ is an $\RR^{d+2}$-realizer of $D$.
\end{proof}

%
%

The following proposition shows that the weak dimension of an induced subdigraph of a digraph cannot exceed that of the digraph.

\begin{Prop}\label{prop:induced dim}
For a digraph $D$ and its induced subdigraph $D'$, $\dim(D') \le \dim(D)$.
\end{Prop}
\begin{proof}
Let $d=\dim(D)$.
Then there is an $\RR^d$-realizer $f$ of $D$.
It is easy to see that the restriction of $f$ to $V(D')$ is an $\RR^d$-realizer of $D'$.
Hence $\dim(D') \le \dim(D)$.
\end{proof}

Proposition~\ref{prop:induced dim} may be false for a subdigraph $D'$ of a digraph $D$ (see Figure~\ref{fig:construct}).

\begin{figure}
\begin{center}
\begin{tikzpicture}[x=1.0cm, y=1.0cm]

    \vertex (b2) at (0,3) [label=left:$v_1$]{};
    \vertex (b3) at (2,2) [label=right:$v_2$]{};

    \vertex (b4) at (0,1) [label=left:$v_3$]{};


    \path
 (b2) edge [->,thick] (b3)
 (b4) edge [->,thick] (b3)
 (b2) edge [->,thick] (b4)

;
 \draw (1, 0.5) node{$D$};
\end{tikzpicture}
\qquad \qquad
\begin{tikzpicture}[x=1.0cm, y=1.0cm]

      \vertex (b2) at (0,3) [label=left:$v_1$]{};
    \vertex (b3) at (2,2) [label=right:$v_2$]{};

    \vertex (b4) at (0,1) [label=left:$v_3$]{};


    \path

 (b4) edge [->,thick] (b3)
 (b2) edge [->,thick] (b4)

;

 \draw (1,0.5) node{$D'$}
	;
\end{tikzpicture}

\end{center}
\caption{The digraph $D'$ is a subdigraph of $D$. Yet, $\dim(D')=3 > 1 = \dim(D)$.}
\label{fig:construct}
\end{figure}
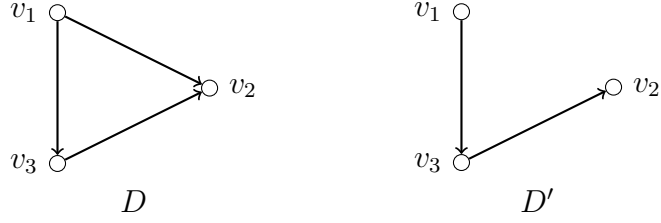

For given two vertex-disjoint digraphs $D_1$ and $D_2$, we denote by $D_1 \cup D_2$ the digraph having the vertex set $V(D_1) \cup V(D_2)$ and the arc set $A(D_1) \cup A(D_2)$.

\begin{Prop}\label{prop:disjoint union}
For a positive integer $k \ge 2$ and $k$ digraphs $D_1$, $D_2$, $\ldots$, $D_k$,
$$d \le \dim(\bigcup_{i=1}^k D_i) \le 2\lfloor (d+1)/2 \rfloor$$
where $d =\max\{\dim(D_1), \ldots, \dim(D_k)\}$.

\end{Prop}
\begin{proof}
Let $\dim(D_i)=d_i$ for each $i \in [k]$.
Then there exists an $\RR^{d_i}$-realizer $f_i$ of $D_i$ for each $i \in [k]$.
Without loss of generality, we may assume $d_1$ is the maximum among $d_1$, $\ldots$, $d_k$.
Since $D_1$ is an induced subdigraph of $\bigcup_{i=1}^k D_i$, $$d_1 \le \dim(\bigcup_{i=1}^k D_i)$$ by Proposition~\ref{prop:induced dim}.

Suppose that $d_1 =0$.
Then $d_2 = \cdots = d_k =0$.
Therefore $D_i$ is an empty digraph for $i=1$, $\ldots$, $k$ by Lemma~\ref{lem:no arc} and so $D$ is an empty digraph.
Thus $\dim(D)=0$ by Lemma~\ref{lem:no arc}.

Now suppose that $d_1 \ge 1$.
To show $\dim(\bigcup_{i=1}^k D_i) \le 2\lfloor (d_1+1)/2 \rfloor$, we define a map $f_i^* : V(D_i) \to \RR^{2\lfloor (d_1+1)/2 \rfloor}$ for each $i \in [k]$ as follow.
When we refer to a component of $f_i$-value which does not really exist, it is assumed that $0$ is assigned to the component.
For example, when we mention $[f_1(v)]_{d_1+1}$ for a vertex $v \in V(D_1)$, we assume that the codomain of $f_1$ is extended to $\RR^{d_1+1}$ by adjoining a component with $0$ at the end of each $d_1$-tuple that is an $f_1$-value.

For notational convenience, let $\Gamma=\lfloor (d_1+1)/2 \rfloor$.

Now let $f_1^*=f_1$.
Then, inductively, for each $i =1$, $\ldots$, $k-1$ and for each $v \in V(D_{i+1})$, we let
$$[f_{i+1}^*(v)]_t =
\begin{cases}
  - m_{i+1,1} +
  M^*_{i,1}+[f_{i+1}(v)]_j+1 & \mbox{for each $t \in \{1, \ldots, \Gamma\}$}   \\
  -M_{i+1,2} + m^*_{i,2}+[f_{i+1}(v)]_j-1 & \mbox{for each $t \in \{\Gamma+1, \ldots, 2\Gamma\}$}
\end{cases}$$
where
\[
m_{i+1,1} = \min\{[f_{i+1}(u)]_j \mid u \in V(D_{i+1}), \text{  $j \in \{1, \ldots, \Gamma\}$}\},
\]
\[
M^*_{i,1} = \max\{[f_i^*(u)]_j \mid u \in V(D_{i}), \text{  $j \in \{1, \ldots, \Gamma\}$}\},
\]
\[
M_{i+1,2} = \max\{[f_{i+1}(u)]_j \mid u \in V(D_{i+1}), \text{  $j \in \{\Gamma+1, \ldots, 2\Gamma\}$}\},
\]
and
\[
m^*_{i, 2}=\min\{[f_i^*(u)]_j \mid u \in V(D_{i}), \text{ $j \in \{\Gamma+1, \ldots, 2\Gamma\}$}\}.
\]
Then, by the definitions of $m_{i+1,1}$, $M_{i+1,2}$, $M^*_{i,1}$, and $m^*_{i,2}$, given $u \in V(D_i)$ and $v \in V(D_{i+1})$,
\[
[f^*_{i+1}(v)]_t > M^*_{i,1} \ge [f^*_{i}(u)]_t
\] for each $i=1$, $\ldots$, $k-1$ and  $t=1$, $\ldots$, $\Gamma$ and
\[
[f^*_{i+1}(v)]_t < m^*_{i,2} \le [f^*_{i}(u)]_t
\] for each $i=1$, $\ldots$, $k-1$ and $t=\Gamma+1$, $\ldots$, $2\Gamma$. Then it follows that, for $u \in V(D_i)$ and $v \in V(D_{i'})$ where $i' > i$,
\begin{equation}\label{eqn:comparison}
[f^*_{i'}(v)]_t > [f^*_{i}(u)]_t \text{ for }t=1, \ldots, \Gamma; \ [f^*_{i'}(v)]_t < [f^*_{i}(u)]_t \text{ for }t=\Gamma+1, \ldots, 2\Gamma.
\end{equation}
Now we define a map $f: \bigcup_{i=1}^k V(D_i) \to \RR^{2\Gamma}$ by
\[f\vert_{V(D_i)}=f_i^*\]
for each $i=1$,$\ldots$, $k$ where $f\vert_{V(D_i)}$ denotes the restriction of $f$ to $V(D_i)$.
To show that $f$ is an $\RR^{2\Gamma}$-realizer of $\bigcup_{i=1}^k D_k$, we first take $u, v \in V(D_{i+1})$ for some $i \in \{1, \ldots, k-1\}$.
Since, given $i=1$, $\ldots$, $k-1$, $m_{i+1, 1}$, $M^*_{i,1}$, $M_{i+1, 2}$, and $m^*_{i,2}$ are constants with respect to vertices in $D_{i+1}$,
\[
[f_{i+1}^*(u)]_t > [f_{i+1}^*(v)]_t \Leftrightarrow [f_{i+1}(u)]_t > [f_{i+1}(v)]_t
\]
for  each $t=1$, $\ldots$, $2\Gamma$ and $u, v \in V(D_{i+1})$.
By definition, $[f_{i}^*(u)]_t=[f(u)]_t$ and $[f_{i}^*(v)]_t=[f(v)]_t$, so $$[f(u)]_t > [f(v)]_t \Leftrightarrow [f_{i}(u)]_t > [f_{i}(v)]_t$$ for each $i=1$, $\ldots$, $k$, each $t=1$, $\ldots$, $2\Gamma$, and $u, v \in V(D_i)$.
Thus we have shown that, for each $i=1$, $\ldots$, $k$ and $u, v \in V(D_i)$, $(u, v)$ is an arc in $\bigcup_{i=1}^k D_i$ if and only if $f(v) \prec f(u)$.
Hence, for two vertices $u$ and $v$ belonging to $D_i$ and $D_{i'}$, respectively, for $1 \le i < i' \le k$, $|\GGG_{f(u)>f(v)}|=|\GGG_{f(v)>f(u)}|=\Gamma$
by \eqref{eqn:comparison} and so we may conclude that $f$ is an $\mathbb{R}^{2\Gamma}$-realizer of $\bigcup_{i=1}^k D_k$.
\end{proof}

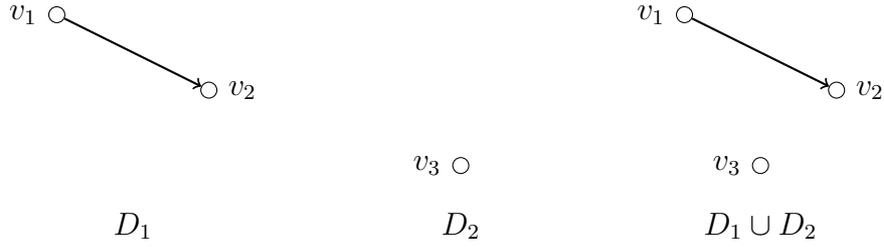
\begin{figure}
\begin{center}
\begin{tikzpicture}[x=1.0cm, y=1.0cm]

    \vertex (b2) at (0,3) [label=left:$v_1$]{};
    \vertex (b3) at (2,2) [label=right:$v_2$]{};


    \path
 (b2) edge [->,thick] (b3)

;
 \draw (1, 0.2) node{$D_1$};
\end{tikzpicture}
\qquad \qquad
\begin{tikzpicture}[x=1.0cm, y=1.0cm]

    \vertex (b4) at (0,1) [label=left:$v_3$]{};


 \draw (0, 0.2) node{$D_2$};
\end{tikzpicture}
\qquad \qquad
\begin{tikzpicture}[x=1.0cm, y=1.0cm]

      \vertex (b2) at (0,3) [label=left:$v_1$]{};
    \vertex (b3) at (2,2) [label=right:$v_2$]{};

    \vertex (b4) at (1,1) [label=left:$v_3$]{};


    \path

 (b2) edge [->,thick] (b3)
;

 \draw (1,0.2) node{$D_1 \cup D_2$}
	;
\end{tikzpicture}

\end{center}
\caption{The digraphs $D_1$, $D_2$, and $D_1 \cup D_2$ have weak majority dimensions $\dim(D_1)=1$, $\dim(D_2)=0$, and $\dim(D_1 \cup D_2)=2= 2\lfloor (\max\{\dim(D_1), \dim(D_2)\}+1)/2\rfloor$, respectively.
These digraphs tell us the upper bound, in Proposition~\ref{prop:disjoint union}, is sharp.}
\label{fig:construct2}
\end{figure}

The upper bound given in Proposition~\ref{prop:disjoint union} is sharp (see Figure~\ref{fig:construct2}).

The Erd\"{o}s-Szekeres lemma given in~\cite{erdos1935combinatorial} states that, for any positive integers $r$ and $s$, every sequence consisting of $rs+1$ distinct real numbers has an increasing subsequence of length $r+1$ or a decreasing subsequence of length $s+1$.
The following lemma is an immediate consequence of the Erd\"{o}s-Szekeres lemma.


\begin{Lem}[\cite{erdos1935combinatorial}]\label{lem:erdos}
Let $S$ be a subset of $\RR^2$ with $|S| = n^2+1$.
Then $S$ contains a chain or an anti-chain of size $n+1$.
\end{Lem}

A digraph $D$ is said to be \emph{transitive} if $(x, z) \in A(D)$ whenever $(x,y) \in A(D)$ and $(y,z) \in A(D)$.

It is known that a majority dimension can be arbitrarily large, which was shown by a probabilistic method.
We give a stronger version which shows that a weak majority dimension can also be arbitrarily large.
By the way, we prove it by constructing a transitive digraph with the weak majority dimension greater than $d$ for each positive integer $d$.
Our result is also meaningful in a following sense.
In 1941, Dushnik and Miller showed that a digraph has majority dimension at most two if and only if it is transitive and its incomparability graph is transitive orientable \cite{bachmeier2017k}.
By looking at their result, one might think that a transitive digraph has a small majority dimension.
Yet, our result shows that a transitive digraph might have an arbitrarily large weak majority dimension, which implies that a transitive digraph might have an arbitrarily large majority dimension.

\begin{Thm}\label{thm:has no upper bound}
There is a transitive digraph $D$ such that $\dim(D) > d$ for a positive integer $d$.
\end{Thm}
\begin{proof}
Let $r=(2d)^{2^{d-1}}+1$ for a sufficiently large integer $d$.
Let $D$ be a digraph with the vertex set
$$ V(D) =[r] \cup \{S \mid S\subset [r] \text{ and } |S|=d+1\}$$
and the arc set
$$A(D) = \left\{(i, S) \mid  i \in S, S\subset [r] \text{ and } |S|=d+1 \right\}.$$
Then it is vacuously true that $D$ is transitive.
In the following, we shall show $\dim(D) > d$.

Suppose, to the contrary, that $\dim(D) \le d$.
Then, by ($\star$), there is an $\RR^d$-realizer of $D$.
Let $f$ be an $\RR^d$-realizer of $D$.
Without loss of generality, we may assume
\begin{equation}\label{eqn:chain}
[f(1)]_1 \ge \cdots \ge [f(r)]_1.
\end{equation}

We endow $\RR^2$ with a partial order $\le_{2}$ defined by $(x_1, x_2) \le_{2} (y_1, y_2)$ if and only if $x_1 \le y_1$ and $x_2 \le y_2$.
Now consider the set $\{([(f(i)]_1, [f(i)]_2) \mid i \in [r]\}$, which is a subset of $\RR^2$.
Since the set has size $r=(2d)^{2^{d-1}}+1$, there exists a subset $I_2 \subset \left[r \right]$ of size $(2d)^{2^{d-2}}+1$ such that $\{([f(i)]_1, [f(i)]_2) \mid i \in I_2 \}$ forms a chain or an anti-chain by Lemma~\ref{lem:erdos}.
Then, by applying Lemma~\ref{lem:erdos} to the set $\{([f(i)]_1, [f(i)]_3) \mid i \in I_2 \}$, there exists a subset $I_{3} \subset I_2$ of size $(2d)^{2^{d-3}}+1$ such that $\{([f(i)]_1, [f(i)]_3) \mid i \in I_{3} \}$ forms a chain or an anti-chain.
We apply Lemma~\ref{lem:erdos} repeatedly in this way until we obtain a set $I_d$ of size $(2d)^{2^0}+1 = 2d+1$ so that $I_d \subset I_{d-1} \subset \cdots \subset I_2 \subset [r]$ and  $\{([f(i)]_1, [f(i)]_t) \mid i \in I_t \}$ forms a chain or an anti-chain in $\RR^2$ for each $t=2$, $3$, $\ldots$, $d$.
We may regard $I_d = [2d+1]$.

Since $\{([f(i)]_1, [f(i)]_t) \mid i \in I_t \}$ forms a chain or an anti-chain for each $t=2$, $\ldots$, $d$, by \eqref{eqn:chain},
\begin{equation}\label{eqn:chain2}
 [f(1)]_t \ge \cdots \ge [f({2d+1})]_t \quad \text{ or }  \quad [f(1)]_t \le \cdots \le [f({2d+1})]_t
\end{equation}
for each $t \in [d]$.

Let $J_1=\GGG_{f(1)>f({2d+1})}$, $J_2=\GGG_{f(2d+1)>f({1})}$, and  $J_3=[d] \setminus (J_1 \cup J_2)$.
Since the vertices $1$ and $2d+1$ are not adjacent in $D$ by definition, $f(1)$ and $f(2d+1)$ are incomparable, and so $|J_1|=|J_2|$.
Then $|J_1|=k$ for some nonnegative integer $k \in \{0, 1, \ldots, \left\lfloor d /2 \right\rfloor \}$.
Then, by definition, $[f(1)]_t = [f({2d+1})]_t$ for each $t \in J_3$.
Thus, by \eqref{eqn:chain2},
\begin{equation}\label{eqn:chain equal J_3}
[f(1)]_t = \cdots = [f({2d+1})]_t
\end{equation}
for each $t \in J_3$.

Now let $S^* = \{1, 3, \ldots, 2i+1, \ldots, 2d+1\}$. Then $S^* \in V(D)$. Let
$$p_{j,i} = |\GGG_{f(i)>f(S^*)} \cap J_j| - |\GGG_{f(S^*)>f(i)}\cap J_j|$$
for each $i \in [2d+1]$ and each $j \in [3]$.
By \eqref{eqn:chain2} and the definition of $J_j$, $|\GGG_{f(i_1)>f(S^*)} \cap J_j| \ge |\GGG_{f(i_2)>f(S^*)} \cap J_j|$ if and only if $|\GGG_{f(S^*)>f(i_1)} \cap J_j| \le |\GGG_{f(S^*)>f(i_2)} \cap J_j|$ for $1 \le i_1, i_2 \le 2d+1$ and $j \in [3]$.
Thus $p_{j, i_1} \ge p_{j, i_2}$ if and only if $|\GGG_{f(i_1)>f(S^*)} \cap J_j| \ge |\GGG_{f(i_2)>f(S^*)} \cap J_j|$ for $1 \le i_1,  i_2 \le 2d+1$ and $j \in [3]$.

Since the  former inequalities in \eqref{eqn:chain2} hold for $t \in J_1$, that is, $[f(q)]_t \ge [f(r)]_t$ for $1 \le q < r \le 2d+1$ and $t \in J_1$.
Thus, for $1 \le q < r \le 2d+1$ and $t \in J_1$,
if $[f(r)]_t \ge [f(S^*)]_t$, then  $[f(q)]_t \ge [f(S^*)]_t$ and so
\begin{equation}\label{eqn:p}
p_{1,1}  \ge p_{1,2}  \ge \cdots \ge p_{1, 2d+1}.
\end{equation}
Similarly, we may show that
\begin{equation}\label{eqn:qr}
p_{2,1} \le p_{2, 2} \le \cdots \le p_{2, 2d+1} \quad \text{and} \quad p_{3,1}= p_{3,2}=\cdots = p_{3,2d+1}.
\end{equation}
Since $\{J_1,J_2,J_3\}$ is a partition of $[d]$,
\begin{equation} \label{eqn:partition}
p_{1,2i-1} + p_{2, 2i-1} + p_{3, 2i-1}=|\GGG_{f(2i-1) > f(S^*)}| - |\GGG_{f(S^*) > f(2i-1)}|
\end{equation}
for each $i \in [d+1]$. Since $({2i-1}, S^*) \in A(D)$,
$$|\GGG_{f(2i-1) > f(S^*)}| - |\GGG_{f(S^*) > f(2i-1)}|>0$$
for each $i \in [d+1]$. Thus
\begin{equation} \label{eqn:one}
p_{1,2i-1} + p_{2, 2i-1} + p_{3, 2i-1} > 0
\end{equation} for each $i \in [d+1]$.

By the way, since $|J_1|=k$, $p_{1,2i-1} \in \{k, k-2, \ldots, -(k-2), -k\}$ for each $i \in [d+1]$.
Therefore there are at most $k+1$ values which are available for $p_{1,2i-1}$ for each $i \in [d+1]$. Yet, $k+1 \le {{d}\over{2}}+1 < d+1$. Thus, by the Pigeonhole principle, there are  $i^*$ and $j^*$ in $[d+1]$ with $j^*>i^*$ such that $p_{1, 2i^*-1}=p_{1, 2j^*-1}$.
By \eqref{eqn:p}, we may assume that $j^*=i^*+1$, i.e.
$p_{1, 2i^*-1}=p_{1, 2i^*+1}$. Then, by \eqref{eqn:p} again, $p_{1, 2i^*-1}=p_{1,2i^*}$. Therefore, by \eqref{eqn:qr} and \eqref{eqn:one},  $$p_{1, 2i^*} + p_{2, 2i^*}+ p_{3, 2i^*} \ge p_{1, 2i^*-1} + p_{2, 2i^*-1}+ p_{3, 2i^*-1} > 0.$$
Hence, by \eqref{eqn:partition}, $|\GGG_{f({2i^*}) > f(S^*)}| - |\GGG_{f({S^*}) > f(2i^*)}| >0$ and so $({2i^*}, S^*) \in A(D)$, which contradicts the definition of $D$.
\end{proof}


\section{Digraph classes with weak majority dimensions at most two}

Now we give necessary and sufficient conditions for a digraph having a weak majority dimension zero and one and  a necessary condition for a digraph having a majority dimension two.

For a digraph $D$, two vertices $u$ and $v$ of $D$ are said to be \emph{homogeneous}, denoted by $u \sim v$, if $N_D^+(u)=N_D^+(v)$ and $N_D^-(u)=N_D^-(v)$ where $N_D^+(w)$ and $N_D^-(w)$ for a vertex $w$ of $D$ stand for the set of out-neighbors and the set of in-neighbors, respectively, of $w$.
Clearly $\sim$ is an equivalence relation on $V(D)$.
We denote the equivalence class containing a vertex $u$ of $D$ by $[u]$.
Let $u_1$, $\ldots$, $u_l$ be vertices of $D$ such that $\{[u_1], \ldots, [u_l] \}$ be a partition of $V(D)$ under the relation $\sim$.
Let $D^*$ be the subdigraph of $D$ induced by $\{u_1, \ldots, u_l\}$.
If $v_1$, $\ldots$, $v_l$ are representatives of $[u_1]$, $\ldots$, $[u_l]$, respectively, then the subdigraph of $D$ induced by $\{v_1, \ldots, v_l\}$ is isomorphic to $D^*$.
In this context, we call a subdigraph of $D$ isomorphic to $D^*$ a \emph{maximally condensed subdigraph} of $D$.
%

\begin{Prop}\label{prop:consider only condensation}
For a digraph $D$ and a maximally condensed subdigraph $D^*$ of $D$, $\dim(D) = \dim(D^*)$.
\end{Prop}
\begin{proof}
Since $D^*$ is an induced subdigraph of $D$, $$\dim(D) \ge \dim(D^*)$$
by Proposition~\ref{prop:induced dim}.
To show the $\dim(D) \le \dim(D^*)$, let $d=\dim(D^*)$ and $f$ be an $\RR^d$-realizer of $D^*$.
We define a map $g:V(D) \to \RR^d$ as follows.
For a vertex $u \in V(D)$, there exists a vertex $v \in V(D^*)$ such that $u \sim v$, and we let
 $g(u)=f(v)$.
Then the map $g$ is obviously well-defined.
To show that $g$ is an $\RR^d$-realizer of $D$, we take two vertices $u$ and $v$ in $D$.
Then there are vertices $u^*$ and $v^*$ in $D^*$ such that $u \sim u^*$ and $v \sim v^*$.
Now
\[
  (u, v) \in A(D) \Leftrightarrow (u^*, v^*) \in A(D^*) \Leftrightarrow f(v^*) \prec f(u^*)
  \Leftrightarrow g(v) \prec g(u).
\]
Hence the map $g$ is an $\RR^d$-realizer of $D$ and so $\dim(D) \le \dim(D^*)$.
\end{proof}

By the above proposition, it is sufficient to consider a maximally condensed subdigraph of a digraph to find its weak majority dimension.


\begin{Prop}\label{prop:dim1}
For a digraph $D$ and a maximally condensed subdigraph $D^*$ of $D$, $\dim(D) = 0$ if and only if $D^*$ is a trivial digraph;  $\dim(D) = 1$ if and only if  $D^*$ is a nonempty acyclic tournament, that is, $A(D^*)$ is a total order on $V(D^*)$.
\end{Prop}
\begin{proof}
By Lemma~\ref{lem:no arc} and Proposition~\ref{prop:consider only condensation}, $\dim(D) = 0$ if and only if $D^*$ is a trivial digraph.

Now we show that $\dim(D) = 1$ if and only if  $D^*$ is a nonempty acyclic tournament.
By Proposition~\ref{prop:consider only condensation}, it is sufficient to show that $\dim(D^*)=1$ if and only if  $D^*$ is a nonempty acyclic tournament.
Let $V(D^*)=\{u_1, u_2, \ldots, u_n\}$.
Suppose that $\dim(D^*)=1$.
Then $D^*$ is nonempty by Lemma~\ref{lem:no arc}, so $n \ge 2$.
Moreover, there is an $\RR$-realizer $f$ of $D^*$.
For $1 \le i < j \le n$, since $u_i$ and $u_j$ belong to distinct equivalence classes under $\sim$,  $f(u_i) \neq f(u_j)$.
Then, without loss of generality, we may assume that $f(u_i) < f(u_j)$ for any $1 \le i < j \le n$.
By the definition of realizer, $(u_j, u_i) \in A(D^*)$ for any $1 \le i <j \le n$, and so $D^*$ is an acyclic tournament.
Therefore the ``only if'' part is true.

Suppose that $D^*$ is a nonempty acyclic tournament.
Then, since an acyclic digraph has an acyclic labeling, without loss of generality, we may assume that $(u_j, u_i) \in A(D^*)$ for any $1 \le i < j \le n$.
Now we define a map $g:V(D^*) \to \RR$ so that $g(u_i)=i$ for any $i \in [n]$.
Then it is easy to check that $g$ is an $\RR$-realizer of $D^*$ and so $\dim(D^*) \le 1$.
By the way, since $D^*$ is nonempty, $\dim(D^*) \ge 1$ by Lemma~\ref{lem:no arc}.
Hence we may conclude that $\dim(D^*)=1$ and so the ``if'' part is true.
\end{proof}

In the rest of this section, we give a necessary condition for a digraph having a weak majority dimension two.
We need the following lemma.

\begin{Lem}\label{lem:directed path of length two}
A directed path of length two has the weak majority dimension three.
\end{Lem}
\begin{proof}
Let $P:=v_1 \rightarrow v_2 \rightarrow v_3$ be a directed path of length two and $f:V(P) \to \RR^3$ be a map defined by
$$
\text{$f(v_1)=(1,2,3)$, $f(v_2)=(3,1,2)$, and $f(v_3)=(2,0,3)$.}
$$
Then it is easy to see that $f$ is an $\RR^3$-realizer of $P$ and so $\dim(P) \le 3$.

Suppose, to the contrary, that $\dim(P) \le 2$.
Then there is an $\RR^2$-realizer $g$ of $P$ by ($\star$).
Since $v_1$ and $v_3$ are not adjacent in $D$, $g(v_1)$ and $g(v_3)$ are incomparable, and so
\begin{equation}\label{eqn:path1}
|\GGG_{g(v_1) > g(v_3)}| = |\GGG_{g(v_3) > g(v_1)}|.
\end{equation}
In addition, since $(v_1, v_2) \in A(P)$ and $(v_2, v_3) \in A(P)$,
\begin{equation}\label{eqn:path2}
\text{$|\GGG_{g(v_1) > g(v_2)}| - |\GGG_{g(v_2) > g(v_1)}| >0$ and $|\GGG_{g(v_2) > g(v_3)}| - |\GGG_{g(v_3) > g(v_2)}| >0$.}
\end{equation}

Since $g(v_1)$, $g(v_2)$, and $g(v_3)$ belong to $\RR^2$, $|\GGG_{g(v_i) > g(v_{i+1})}|+|\GGG_{g(v_{i+1}) > g(v_i)}| \le 2$ for each $i=1, 2$.
By \eqref{eqn:path2}, $|\GGG_{g(v_{i+1}) > g(v_{i})}|=0$ for each $i=1, 2$.
Therefore
\begin{equation}\label{eqn:path3}
[g(v_1)]_1 \ge [g(v_2)]_1 \ge [g(v_3)]_1 \quad \text{and} \quad [g(v_1)]_2 \ge [g(v_2)]_2 \ge [g(v_3)]_2.
\end{equation}
Thus $|\GGG_{g(v_3) > g(v_1)}|=0$.
If $|\GGG_{g(v_1) > g(v_3)}|=0$, then the inequalities in \eqref{eqn:path3} become equalities, which contradicts \eqref{eqn:path2}.
Therefore $|\GGG_{g(v_1) > g(v_3)}|>0$ and so $|\GGG_{g(v_1) > g(v_3)}| - |\GGG_{g(v_3) > g(v_1)}| >0$, which contradicts \eqref{eqn:path1}.
Hence $\dim(P) \ge 3$ and we may conclude that $\dim(P)=3$.
\end{proof}

\begin{Thm}\label{thm:transitive}
Every digraph with weak majority dimension at most two is transitive.
\end{Thm}
\begin{proof}
Take a digraph $D$ with $\dim(D) \le 2$.
If $D$ contains a directed path of length two as an induced subdigraph, then $\dim(D) \ge 3$ by Proposition~\ref{prop:induced dim} and Lemma~\ref{lem:directed path of length two}, and we reach a contradiction.
Thus $D$ does not contain a directed path of length two as an induced subdigraph and hence the statement of this theorem is vacuously true.
\end{proof}

While proving Theorem~\ref{thm:transitive}, we also have shown the following statement.

\begin{Cor}\label{dim(P_2)=3}
Every digraph with weak majority dimension at most two does not contain a directed path of length two as an induced subdigraph.
\end{Cor}

The converse of Theorem~\ref{thm:transitive} is not true, i,e. there is a transitive digraph with weak majority dimension greater than two by Theorem~\ref{thm:has no upper bound}.

\section{Weak majority dimensions of directed paths and directed cycles}
In this section, we study weak majority dimensions of directed paths and directed cycles.
To do so, we need following lemmas which can easily be checked.

\begin{Lem}\label{lem:odd dim, has odd same comp}
Let $x \in \RR^{2d+1}$ and $y \in \RR^{2d+1}$ for a positive integer $d$.
If $x$ and $y$ are incomparable, then the size of the set $\{i \in [2d+1] \mid [x]_i = [y]_i \}$ is odd.
\end{Lem}

\begin{Lem}\label{lem:have no same component}
For $x, y, z \in \RR^3$, suppose that $x \prec y$,  $y \prec z$ and $x$ and $z$ are incomparable.
Then $\{i \in [3] \mid [x]_i = [y]_i \}=\{i \in [3] \mid [y]_i = [z]_i\}=\emptyset$.
\end{Lem}
\begin{proof}
Suppose, to the contrary, that $\{i \in [3] \mid [x]_i = [y]_i \} \neq \emptyset$.
Without loss of generality, $[x]_1=[y]_1$.
If $[x]_2 > [y]_2$ or $[x]_3 > [y]_3$, then $|\GGG_{x>y}| \ge 1$ and $|\GGG_{y>x}| \le 1$, which contradicts to the fact that $x \prec y$.
Therefore $[x]_2 \le [y]_2$ and $[x]_3 \le [y]_3$.

Since $y \prec z$, $|\GGG_{z>y}| - |\GGG_{y>z}| > 0$.
Since $[x]_i \le [y]_i$ for each $i=1,2,3$, $\GGG_{z>y} \subset \GGG_{z>x}$ and $\GGG_{y>z} \supset \GGG_{x>z}$.
Hence $$|\GGG_{z>x}| - |\GGG_{x>z}| \ge |\GGG_{z>y}| - |\GGG_{y>z}| > 0$$ and we reach a contradiction to the hypothesis that $x$ and $z$ are incomparable.
Therefore $\{i \in [3] \mid [x]_i = [y]_i \}=\emptyset$.
By applying a symmetric argument, we may show that $\{i \in [3] \mid [y]_i = [z]_i\}=\emptyset$.
\end{proof}

\begin{Thm}\label{thm:directed path dim is bounded}
Every directed path has weak majority dimension at most four and there exist directed paths with weak majority dimension four.
\end{Thm}
\begin{proof}
Let $n$ be a positive integer and $P:=v_1 \rightarrow v_2 \rightarrow \cdots \rightarrow v_{2n-1}$ be a directed path of length $2n-2$.
To show the $\dim(P) \le 4$, we define a map $f:V(P) \to \RR^4$ as follow:
\begin{itemize}
  \item[(i)] $f(v_{2i-1})=(2n-2i+1,2n-2i+1,2i-1, 2i-1)$ for each $i =1$, $\ldots$, $n$;
  \item[(ii)] $f(v_{2i})=(2n-2i, 2n-2i, 2i-2, 2i+2)$ for each $i=1$, $\ldots$, $n-1$.
\end{itemize}
Then, for any $1 \le i < j \le 2n-1$ with $j-i \ge 2$, $$\GGG_{f(v_{i})>f(v_{j})}=\{1,2\} \quad \text{and} \quad  \GGG_{f(v_{j})>f(v_{i})}=\{3,4\}$$
and so $f(v_i)$ and $f(v_j)$ are incomparable.
On the other hand, $$\GGG_{f(v_{2i-1})>f(v_{2i})}=\{1,2,3\} \quad \text{and} \quad  \GGG_{f(v_{2i})>f(v_{2i+1})}=\{1,2,4\}$$ for any $i \in [n-1]$.
Therefore $f$ is an $\RR^4$-realizer of $P$ and so $\dim(P) \le 4$.
Since $n$ is arbitrarily chosen, the weak majority dimension of a directed path of even length is at most four and so, by Lemma~\ref{prop:induced dim}, the weak majority dimension of a directed path is at most four.

Now we shall show that a directed path
\[Q:=u_1 \rightarrow u_2 \rightarrow \cdots \rightarrow u_{10}\]
has weak majority dimension four.
Suppose, to the contrary, that $\dim(Q) \le 3$.
Then there is an $\RR^3$-realizer $g$ of $Q$ by ($\star$).

Now take $i$ and $j$ in the set $\{1, 2, \ldots, 10\}$ with $j-i \ge 2$.
Then, since $u_i$ and $u_j$ are not adjacent in $Q$, $g(u_i)$ and $g(u_j)$ are incomparable.
Thus, by Lemma~\ref{lem:odd dim, has odd same comp}, $|\{l \in [3] \mid [g(u_i)]_l = [g(u_j)]_l \}|$ is odd.
Since $N_Q^+(u_i) \neq N_Q^+(u_j)$, $g(u_i) \neq g(u_j)$.
Therefore we have shown that
\begin{equation} \label{eq:one}
|\{l \in  [3]  \mid [g(u_i)]_l = [g(u_j)]_l \}|=1
\end{equation} for each pair of $i$ and $j$ in the set $\{1, 2, \ldots, 10\}$ with $j-i \ge 2$.

Let $$I_l= \{i \in \{3,4,\ldots,9\} \mid  [g(u_1)]_l=[g(u_i)]_l \}$$ for each $l = 1, 2, 3$.
Since $i-1 \ge 2$ for $i \in \{3, 4, \ldots, 9\}$, the set $\{3, 4, \ldots, 9\}$ is the disjoint union of $I_1$, $I_2$, and $I_3$ by \eqref{eq:one}, so $$|I_1|+|I_2|+|I_3|=7.$$
Therefore, by the Pigeonhole principle, at least one of $|I_1|$, $|I_2|$, and $|I_3|$ is greater than two.
Without loss of generality, we may assume that $|I_1| \ge 3$.
Then we may take $p, q, r \in I_1$ satisfying $3 \le p<q<r \le 9$.
Since $p, q, r \in I_1$, \begin{equation}\label{eq:equal}[g(u_1)]_1 =[g(u_p)]_1=[g(u_q)]_1=[g(u_r)]_1.\end{equation}
By Lemma~\ref{lem:have no same component},
\begin{equation}\label{eq:equal2}
\{j \in [3] \mid [g(u_k)]_j = [g(u_{k+1})]_j \}=\emptyset
\end{equation}
for each $k \in \{1, 2, \ldots, 9\}$.
Therefore $|p-q| \ge 2$, $|q-r| \ge 2$, and $|r-p| \ge 2$.
Thus any pair of $g(u_p)$, $g(u_q)$, and $g(u_r)$ are incomparable.
Then, by \eqref{eq:one} and \eqref{eq:equal}, $[g(u_1)]_2$, $[g(u_p)]_2$, and $[g(u_q)]_2$ are all distinct and $[g(u_1)]_3$, $[g(u_p)]_3$, and $[g(u_q)]_3$ are all distinct.
By the way, for each $i =1, p, q$,
\[
(r+1)-i \ge 2,
\]
so
\[
|\{l \in  [3]  \mid [g(u_{r+1})]_l = [g(u_i)]_l \}|=1
\]
by \eqref{eq:one}.
Yet, $[g(u_r)]_1 \neq [g(u_{r+1})]_1$ by \eqref{eq:equal2}.
Therefore $$[g(u_1)]_1 =[g(u_p)]_1=[g(u_q)]_1 \neq [g(u_{r+1})]_1$$ by \eqref{eq:equal}.
Since $[g(u_1)]_2$, $[g(u_p)]_2$, and $[g(u_q)]_2$ are all distinct and $[g(u_1)]_3$, $[g(u_p)]_3$, and $[g(u_q)]_3$ are all distinct, $[g(u_{r+1})]_l$ is distinct from at least two of $[g(u_1)]_l$, $[g(u_p)]_l$, and $[g(u_q)]_l$ for each $l=2,3$.
Therefore $[g(u_{r+1})]_l \neq [g(u_{i^*})]_l$ for some $i^* \in \{1, p, q\}$ and for each $l=2,3$ and we reach a contradiction to \eqref{eq:one}.
\end{proof}

By Proposition~\ref{prop:add arc} and Theorem~\ref{thm:directed path dim is bounded}, the weak majority dimension of a directed cycle with length at least three is at most six.
To improve this upper bound, we need the following lemma.

\begin{Lem}\label{lem:cycle matrix}
For a positive integer $n \ge 4$, there is an $n \times 4$ matrix $A_n=(a_{ij})$ satisfying the following properties.
\begin{itemize}
  \item[(i)] $a_{ij} \in \NN$ for any $i \in [n]$ and $j \in [4]$;
  \item[(ii)] $a_{ik} \neq a_{jk}$ for any $1 \le i < j \le n$ and $k \in [4]$;
  \item[(iii)] $a_{nj}=\max\{a_{lj} \mid l \in [n]\}$ and $a_{1k}=\min\{a_{lk} \mid l \in [n]\}$ for some $j, k \in [4]$ and $j \neq k$;
  \item[(iv)] $|\{ k\in [4] \mid a_{ik}-a_{(i+1)k}>0 \}|=3$ for any $i \in [n]$ (we identify $n+1$ with $1$);
  \item[(v)] $|\{ k\in[4] \mid a_{ik}-a_{jk}>0 \}|=2$ for any $i, j \in [n]$ with $|i -j| \ge 2$ and $\{i, j\}\neq\{1,n\}$.
\end{itemize}
\end{Lem}
\begin{proof}
We show by induction on $n$.
Let
\[
A_4=\begin{bmatrix}
 3 & 1 & 2 & 4   \\
 2 & 4 & 1 & 3   \\
 1 & 3 & 4 & 2   \\
 4 & 2 & 3 & 1
\end{bmatrix}.
\]
Then it is easy to check that $A_4$ satisfies the conditions given in this lemma.

Now suppose that an $n \times 4$ matrix $A_n$ satisfies the given conditions for some integer $n \ge 4$.
Since a matrix obtained from $A_n$ by a column permutation still satisfies the given conditions, we may assume that $a_{n1}=\max\{a_{l1}\mid l \in [n]\}$ and $a_{12}=\min\{a_{l2} \mid l \in [n]\}$.
We let $M=\max\{a_{l3} \mid l \in [n]\}$ and  $m=\min\{a_{l4} \mid l \in [n]\}$.
Now we define the $(n+1) \times 4$ matrix $A_{n+1}=(a^*_{ij})$ as follow:
\begin{itemize}
  \item[(1)] $a^*_{ij}=2a_{ij}$ for any $i \in [n]$ and any $j \in [4]$;
  \item[(2)] $a^*_{n+1,1}=2a_{n1}-1$, $a^*_{n+1,2}=2a_{12}+1$, $a^*_{n+1,3}=2M+1$, and $a^*_{n+1,4}=2m-1$.
\end{itemize}
Obviously $A_{n+1}$ satisfies the conditions (i) and (ii) given in the lemma statement.
Then $a^*_{n+1,3}$ and $a^*_{12}$ are the maximum and the minimum of the third column and the second column, respectively, in $A_{n+1}$ and $A_{n+1}$ satisfies the condition (iii).
It is easy to check the following:
$\{k\in [4]  \mid a^*_{nk}-a^*_{n+1,k}>0 \} =\{1,2,4\}$; $\{k\in [4] \mid   a^*_{n+1,k}-a^*_{1k}>0 \}=\{1,2,3\}$;
$\{k\in [4] \mid   a^*_{n+1,k}-a^*_{jk}>0\}=\{1,3\}$; $\{k\in [4] \mid   a^*_{jk}-a^*_{n+1,k}>0\}=\{2,4\}$ for each $j =2$, $\ldots$, $n-1$.
Thus the conditions (iv) and (v) are satisfied.
\end{proof}
 %
%

\begin{Thm}\label{thm:directed cycle dim is bounded}
Every directed cycle has weak majority dimension three or four both of which are achievable.
\end{Thm}
\begin{proof}
By the assumption at the beginning of the paper, we only consider directed cycles of length greater than or equal to $3$.
Let $n$ be an integer greater than or equal to $3$ and $C_n:=v_1 \rightarrow v_2 \rightarrow \cdots \rightarrow v_n \rightarrow v_1$ be a directed cycle of length $n$.
Since a directed cycle is not transitive,  $\dim(C_n) \ge 3$ by Theorem~\ref{thm:transitive}.

Now we show $\dim(C_n) \le 4$.
Suppose $n=3$ and let $f:V(C_3) \to \RR^3$ be a map defined by
$$f(v_1)=(1,2,3), \quad f(v_2)=(3,1,2), \quad \text{and} \quad f(v_3)=(2,3,1).$$
Then it is easy to check that $f$ is an $\RR^3$-realizer of $C_3$ and so $\dim(C_3) \le 3$.
Now suppose $n \ge 4$.
Then, by Lemma~\ref{lem:cycle matrix}, there is an $n \times 4$ matrix $A_n=(a_{ij})$ satisfying the conditions given in the lemma statement.
We define a map $g:V(C_n) \to \RR^4$ by  $$[g(v_i)]_j = a_{ij}$$ for any $i \in [n]$ and $j \in \{1,2,3,4\}$.
Then, by the conditions (ii) and (iv) in Lemma~\ref{lem:cycle matrix}, $|\GGG_{g(v_i)>g(v_{i+1})}|-|\GGG_{g(v_{i+1})>g(v_{i})}|=2$ for each $i=1$, $\ldots$, $n$ (we identify $n+1$ with $1$).
Furthermore, by the conditions (ii) and (v) in Lemma~\ref{lem:cycle matrix}, $|\GGG_{g(v_i)>g(v_j)}|=|\GGG_{g(v_j)>g(v_i)}|$ for each pair of $i$ and $j$ in $[n]$ with $|i-j| \ge 2$ and $\{i, j\}\neq\{1,n\}$.
Thus $g$ is an $\RR^4$-realizer of $C_n$ and so $\dim(C_n) \le 4$.
Therefore the weak majority dimension of $C_n$ is three or four.

In the above argument, we actually showed that the weak majority dimension of $C_3$ is three.
By Theorem~\ref{thm:directed path dim is bounded}, there is a directed path $P$ with $\dim(P) = 4$.
Let $C$ be a directed cycle containing a directed path $P$ as an induced subdigraph.
Then, by Proposition~\ref{prop:induced dim}, $\dim(C) \ge 4$, so $\dim(C)=4$.
\end{proof}

\section{Concluding Remarks}

We gave a necessary condition for a digraph having weak digraph dimension at most two.
We would like to see whether or not there is a meaningful necessary and sufficient condition for a digraph having weak digraph dimension at most two.

By Theorem~\ref{thm:has no upper bound},  for any nonnegative integer $d$, there is a (transitive) digraph $D$ such that $\dim(D) > d$.
We go further to know whether or not the following is true.
\begin{Conj}
For any nonnegative integer $d$, there is a (transitive) digraph $D$ such that $\dim(D) = d$.
\end{Conj}

\section{Acknowledgement}
This research was supported by
the National Research Foundation of Korea(NRF) funded by the Korea government(MEST) (NRF-2017R1E1A1A03070489) and by the Korea government(MSIP) (2016R1A5A1008055).

\end{document}